\documentclass[12pt, reqno]{amsart}
\usepackage{amssymb,dsfont}
\usepackage{eucal}
\usepackage{amsmath}
\usepackage{amscd}
\usepackage[dvips]{color}
\usepackage{multicol}
\usepackage[all]{xy}           
\usepackage{graphicx}
\usepackage{color}
\usepackage{colordvi}
\usepackage{xspace}
\usepackage{txfonts}
\usepackage{lscape}
\usepackage{tikz} 
\usepackage{bm}

\numberwithin{equation}{section}
\usepackage[shortlabels]{enumitem}
\usepackage{ifpdf}
\ifpdf
 \usepackage[colorlinks,final,backref=page,hyperindex]{hyperref}
\else
\fi
\usepackage{tikz}
\usepackage[active]{srcltx}

\makeatletter

\newtheorem{theorem}{Theorem}[section]
\newtheorem{definition}[theorem]{Definition}
\newtheorem{lemma}[theorem]{Lemma}
\newtheorem{proposition}[theorem]{Proposition}
\newtheorem{remark}[theorem]{Remark}

\numberwithin{equation}{section}

\def\({\left(}
\def\){\right)}
\def\al{\alpha}
\def\be{\beta}
\def\dt{\delta}
\def\gm{\gamma}
\def\vf{\varphi}
\def\g{\mathfrak{g}}
\def\U{\mathcal{U}}
\newcommand{\C}{\mathbb C}
\newcommand{\N}{\mathbb{N}}
\newcommand{\Z}{\mathbb{Z}}
\newcommand{\spanc}[1]{\mathrm{Span}_{\C}\left\{#1\right\}}
\newcommand{\ov}[1]{\overline{#1}}

\def\im{\ov{\bm1}}
\def\T{{\mathbb{T}}}
\def\M{\mathbb M}
\def\bmd{{\bm d}}
\def\bme{{\bm e}}
\def\mtvad{{M}_\T\(V,a,\bmd\)}
\def\mftvad{{M}^f_\T\(V,a,\bmd\)}
\def\mgvadp{{M}_\g\(V,a,\bmd,P\)}
\def\mggvadp{M^g_\g\(V,a,\bmd,P\)}
\def\mtwbe{{M}_\T\(W,b,\bme\)}
\def\mgwbeq{{M}_\g\(W,b,\bme,Q\)}
\def\mind{\min\{\bmd\}}
\def\mine{\min\{\bme\}}
\def\mindp{\min\{\bmd,P\}}

\def\td{\T_{\mind}}
\def\te{\T_{\mine}}
\def\tdp{\T_{\mindp}}
\renewcommand{\P}{\mathcal P}
\def\Q{\mathcal Q}

\newcommand{\eleid}[2]{I_{#1+d_{#2}}}
\newcommand{\opwijlm}[1]{\Omega_{l,m}^{(i,j,#1)}}

\begin{document}
\title[Weight modules over Heisenberg-Virasoro and gap-$p$ Virasoro algebras            ]{A class of weight modules over the twisted Heisenberg-Virasoro algebra and gap-$p$ Virasoro algebras}
  \author{Chengkang Xu}
  \address{C. Xu: School of Mathematical Sciences, Shangrao Normal Collage, Shangrao, Jiangxi, P. R. China}                                                                     \email{xiaoxiongxu@126.com}
  \author{Fen Zhang}
  \address{F. Zhang: School of Mathematical Sciences, Shangrao Normal Collage, Shangrao, Jiangxi, P. R. China}                                                                     \email{1024866868@qq.com, Corresponding author}

\date{}
\maketitle

\begin{abstract}
In this paper, we construct a class of simple weight modules over the twisted Heisenberg-Virasoro algebra and gap-$p$ Virasoro algebras from restricted modules over some positive part subalgebra of the twisted Heisenberg-Virasoro algebra.
These modules are new.
In particular, when $p=2$, the gap-$p$ Virasoro algebra is the mirror Heisenberg-Virasoro algebra and we obtain many new simple weight modules for the mirror Heisenberg-Virasoro algebra.
\end{abstract}

\section{Introduction}

Representation theory of the twisted Heisenberg-Virasoro algebra $\T$ has been extensively studied for the past decades.
In \cite{LZ1}, irreducible Harish-Chandra modules were classified for $\T$. They are modules of intermediate series, highest weight modules or lowest weight modules.
Tensor products of modules of intermediate series and the highest weight modules were studied in \cite{LZ2}.
Non-weight modules over $\T$ have also been considered by many mathematicians.
Generalized oscillator representations were constructed using restricted modules over the Heisenberg subalgebra in \cite{LZ2}.
Restricted modules over $\T$ were classified in \cite{CG1,Gao} and \cite{TYZ} separately.
They are essentially tensor products of restricted modules over the Virasoro algebra and restricted modules over the Heisenberg algebra,
or modules induced from some finite dimensional solvable subquotients of $\T$.

The mirror Heisenberg-Virasoro algebra was first studied in \cite{Ba} and named in \cite{LPXZ}. Its structure is similar to the twisted Heisenberg-Virasoro algebra.
In \cite{LPXZ}, the authors classified irreducible Harish-Chandra modules for the mirror Heisenberg-Virasoro algebra $\M$,
which are modules of intermediate series, highest weight modules or lowest weight modules.
They also classified irreducible restricted modules of level zero for $\M$.
Irreducible restricted modules of nonzero-level for $\M$ are classified in \cite{TYZ}.
Like the twisted Heisenberg-Virasoro algebra,
irreducible restricted modules for $\M$ are some tensor product modules or modules induced from some finite dimensional solvable subquotients of $\M$.
Tensor products of modules of intermediate series and highest weight modules of $\M$ were studied in \cite{GZ}.

Let $p>1$ be an integer.
The gap-$p$ Virasoro algebra $\g$ was first introduced in \cite{Xu}, where irreducible Harish-Chandra modules were classified.
They are modules of intermediate series, highest weight modules or lowest weight modules.
The algebra $\g$ has a close relation to the twisted Heisenberg-Virasoro algebra and the algebra of derivations over a rational quantum torus.
Moreover, when $p=2$, $\g$ is isomorphic to the mirror Heisenberg-Virasoro algebra $\M$ (see Section \ref{sec5}).
Irreducible restricted modules of level zero are classified in \cite{GX}.
They are modules induced from some finite dimensional solvable subquotients of $\g$.
Recently, some non-weight $\g$-modules were classified in \cite{XCT},
including Whittaker modules and $\U(\C L_0)$-free modules.

In \cite{CHS}, a class of simple weight modules over the twisted Heisenberg-Virasoro algebra $\T$ with trivial central actions were constructed from irreducible restricted modules over some positive part subalgebras of $\T$.
In this paper, we use the same mechanism to construct simple weight modules with trivial central actions for the twisted Heisenberg-Virasoro algebra $\T$, the gap-$p$ Virasoro algebras (in particular, for the mirror Heisenberg-Virasoro algebra $\M$) from irreducible restricted modules over some positive part subalgebras of $\T$.
The modules constructed here for $\T$ generalize the ones in \cite{CHS}.
We determine the irreducibilities of these newly constructed modules and their isomorphism classes.
We also construct some non-weight modules for $\T$ and $\g$ by using a twisting technique
from \cite{LGZ}.

Let us now briefly describe how this paper is organized.
In Section 2, we recall the algebras $\T$ and $\g$, and introduce some useful operators in $\U(\g)$ and $\U(\T)$.
In Section 3, we give a construction of the above-mentioned modules for $\T$ and $\g$, and prove that these modules are irreducible under a minor condition.
Section 4 is dedicated to determine the isomorphism classes of the modules over $\T$ and $\g$,
and Section 5 is to demonstrate that these modules are new.
In the last section, we use the modules obtained in Section 3 to construct some non-weight modules for $\T$ and $\g$.

Throughout this paper, the symbols $\Z$, $\N$, $\Z_+$, $\C$, and $\C^\times$ are to denote the sets of integers, non-negative integers, positive integers, complex numbers, and nonzero complex numbers, respectively.
All vector spaces and algebras are over $\C$ and the universal enveloping algebra of a Lie algebra $\mathcal L$ is denoted by $\U(\mathcal L)$.
Finally, we fix a positive integer $p>1$ for this paper,
and for any $k\in\Z$, denote by $\ov k$ the residue of $k$ by $p$.

\section{Preliminaries}
In this section, we recall the twisted Heisenberg-Virasoro algebra and gap-$p$ Virasoro algebra, and introduce some related notations.

\begin{definition}
The \textbf{twisted Heisenberg-Virasoro} is a complex Lie algebra $\T$ with a basis
$$\left\{\,L_n,I_n, C_L, C_{LI}, C_I\, \mid\, n\in\Z\,\right\}$$
and Lie brackets
\begin{equation} \label{HV}
 \begin{aligned}
  &[L_m,L_n]=(n-m)L_{m+n}+\frac1{12}(m^3-m)C_L\dt_{m+n,0}, \\
  &[L_m,I_n]=nI_{m+n}+(m^2+m)C_{LI}\dt_{m+n,0}, \ \quad
   [I_m,I_n]=mC_I\dt_{m+n,0},
 \end{aligned}
\end{equation}
for $m,n\in\Z$ and $C_L, C_{LI}, C_I$ being central.
\end{definition}

For any $n\in\N$, the subspace
\[\T_n=\spanc{L_k,I_{n+k}\,\mid\,k\in\N}\]
is a subalgebra of $\T$.
\begin{definition}\label{Restrict}
A module $V$ over the Lie algebra $\T_n$ is called \textbf{restricted}
if for any $v\in V$, there exists $N\in\N$ such that $L_m v=I_mv=0$ for all $m\geq N.$
Let $h,q\in\N$ be such that $h\geq n$.
We say $V$ has the \textbf{annihilating level} $(h,q)$ if
$I_hV\neq 0, L_qV\neq 0$ and $I_kV=L_lV=0$ for all $k>h, l>q$.
\end{definition}

\begin{remark}\label{rmk:Ih}
 (1) Let $V$ be an irreducible $\T_n$-module with annihilating level $(h,q)$.
 Then $I_h$ acts on $V$ as a linear isomorphism (see \cite[Lemma 3.1]{CHS}).\\
 (2) An example of irreducible restricted $\T_0$(or $\T_1$)-module of a fixed annihilating level was given in \cite[Example 3.4]{CHS}, and
 a general construction of such modules was presented in \cite[section 4]{CG1}.
\end{remark}

The gap-$p$ Virasoro algebra may be realized as the universal central extension of some subalgebra of the centerless twisted Heisenberg-Virasoro algebra
\[\T'=\spanc{L_k,I_k\,\mid\,k\in\Z}\]
satisfying the Lie brackets in \eqref{HV} with $C_L, C_{LI}, C_I$ omitted (see \cite{Xu}).

\begin{definition}
The \textbf{gap-$p$ Virasoro algebra} is a complex Lie algebra $\g$ with a basis
$$\left\{L_{pn},I_{pn+i}, C_j\, \mid \, n\in\Z, 1\leq i\leq p-1, 0\leq j\leq [\frac p2] \right\}$$
satisfying the following Lie brackets
\begin{align*}
  &[L_{pm},L_{pn}]=p(n-m)L_{pm+pn}+\frac1{12}(m^3-m)C_0\dt_{m+n,0},            \\
  &[I_{pm+i},I_{pn+j}]=(pm+i)\dt_{i+j,p}\dt_{m+n+1,0} C_{\min\{i,p-i\}},
  \ \ \ \ \ [C_k, \g]=0,             \\
  &[L_{pm},I_{pn+i}]=(pn+i)I_{pm+pn+i},
\end{align*}
where $m,n\in\Z$, $1\leq i,j\leq p-1$, $0\leq k\leq [\frac p2]$, and
$[\frac p2]$ denotes the largest integer that is less than or equal to $\frac p2$.
\end{definition}

\begin{remark}
(1) When $p=2$, the gap-2 Virasoro algebra is exactly the mirror Heisenberg-Virasoro algebra introduced in \cite{Ba} and named in \cite{LPXZ}.\\
(2) Since the modules studied in this paper have trivial central actions,
we may consider $\g$ as a subalgebra of $\T$ by neglecting the central elements
$C_L,C_{LI},C_I, C_i, 0\le i\le [\frac p2]$.
This is also the reason for us to denote elements in $\T$ and $\g$ in a same fashion.
\end{remark}

\begin{definition}
A module $V$ over the Lie algebra $\T$ or $\g$ is called a \textbf{weight module}
if $V$ can be decomposed as a direct sum of eigenspaces under the $L_0$-action.
\end{definition}

To end this section we introduce some operators in $\U(\g)$ and $\U(\T)$  for later use.
For any $s\in\N, l,m\in\Z$ and $1\le i,j\le p-1$, define
\begin{equation}\label{defopw}
\opwijlm s=\sum_{k=0}^s(-1)^{s-k}\binom sk I_{pl-pm-pk+i} I_{pm+pk+j}.
\end{equation}
Clearly, $\opwijlm 0=I_{pl-pm+i} I_{pm+j}$ and it is easy to prove the following recursive formula
\begin{equation}\label{recformula}
\opwijlm s=\Omega_{l,m+1}^{(i,j,s-1)}-\opwijlm {s-1}\quad\text{ for }s>0.
\end{equation}

\section{Constructing weight modules}\label{sec3}

In this section we construct some weight modules over $\T$ and $\g$ from restricted $\T_0$ (or $\T_1$)-modules, and determine their irreducibility.

Fix a $p$-tuple $\bmd=(d_0,d_1,\dots, d_{p-1})\in\{0,1\}^p$
and a subset $P$ of $\{0,1,\dots, {p-1}\}$ in this section.
Denote
\[\mind=\min\{d_0,d_1,\dots, d_{p-1}\}\quad\text{ and }\quad
   \mindp=\min\{d_i\mid i\in P\}.\]
Then we have subalgebras $\td$ and $\tdp$ of $\T$, and $\tdp\subseteq \td$.

Let $a\in\C$, $V$ be a restricted $\td$-module and let $\C[x^{\pm1}]$ denote the Laurent polynomial ring in the indeterminant $x$.
We define a $\T$-action on the space $V\otimes\C[x^{\pm1}]$ by
\begin{equation} \label{T-action}
\begin{aligned}
 &L_m\cdot v(k)=\(a+k+\sum_{j=0}^\infty\frac{m^{j+1}}{(j+1)!}L_j\)v(m+k);\\
 &I_{pm+i}\cdot v(k)=\sum_{j=0}^\infty\frac{(pm+i)^{j+d_i}}{(j+d_i)!}\eleid ji v(pm+i+k);\\
 &C_L\cdot v(k)=C_{LI}\cdot v(k)=C_I\cdot v(k)=0,
\end{aligned}
\end{equation}
where $m,k\in\Z, v\in V, 0\le i\le p-1$ and $v(k)=v\otimes x^k$.

Denote
\[\C[P]=\bigoplus_{r\in P}x^r\C[x^{\pm p}]\subseteq\C[x^{\pm1}]\quad\text{ and }\quad
   \P=\left\{k\in\Z\mid\ov k\in P\right\}
     =\left\{k\in\Z\mid x^k\in \C[P]\right\}.\]
Consider $V$ as a $\tdp$-module and
we define a $\g$-action on the space $V\otimes\C[P]$ by
\begin{equation} \label{g-action}
\begin{aligned}
 &L_{pm}\cdot v(k)=\(a+k+\sum_{j=0}^\infty\frac{(pm)^{j+1}}{(j+1)!}L_j\)v(pm+k);\\
 &I_{pm+i}\cdot v(k)=
   \dt_{\ov{i+k},P}\sum_{j=0}^\infty\frac{(pm+i)^{j+d_i}}{(j+d_i)!}\eleid ji v(pm+i+k);\\
 &C_l\cdot v(k)=0,
\end{aligned}
\end{equation}
where $m\in\Z, k\in\P, v\in V, 1\le i\le p-1,0\le l\le[\frac p2]$ and
$\dt_{j,P}$ is the Kronecker symbol defined by
\[\dt_{j,P}=\begin{cases}
               1 &\text{ if }j\in P;\\
               0 &\text{ if }j\notin P.
            \end{cases}
\]

Since $V$ is restricted, the actions in \eqref{T-action} and \eqref{g-action}
contain only finitely many terms, hence make sense.
Furthermore, we have the following
\begin{proposition}
(1) The actions in \eqref{T-action} make $V\otimes\C[x^{\pm1}]$ a weight $\T$-module.\\
(2) The actions in \eqref{g-action} make $V\otimes\C[P]$ a weight $\g$-module.
\end{proposition}
\begin{proof}
(1) For any $k,m,n\in\Z, 0\le i,j\le p-1$ and $v\in V$, it is clear that
\[\(I_{pm+i}I_{pn+j}-I_{pn+j}I_{pm+i}\)\cdot v(k)=0,\]
and we have
\begin{align*}
   &\(L_mI_{pn+i}-I_{pn+i}L_m\)\cdot v(k)\\
 =&\(a+pn+i+k+\sum_{q=0}^\infty\frac{m^{q+1}}{(q+1)!}L_q\)
         \(\sum_{l=0}^\infty\frac{(pn+i)^{l+d_i}}{(l+d_i)!}\eleid li\)v(m+pn+i+k)\\
  &-\(\sum_{l=0}^\infty\frac{(pn+i)^{l+d_i}}{(l+d_i)!}\eleid li\)
       \(a+k+\sum_{q=0}^\infty\frac{m^{q+1}}{(q+1)!}L_q\)v(m+pn+i+k)\\
 =&(pn+i)\sum_{l=0}^\infty\frac{(pn+i)^{l+d_i}}{(l+d_i)!}\eleid li v(m+pn+i+k)\\
  &+\sum_{l=0}^\infty\sum_{q=0}^\infty\frac{(pn+i)^{l+d_i}m^{q+1}}{(l+d_i)!(q+1)!}
     [L_q,\eleid li]v(m+pn+i+k)\\
 =&(pn+i)\sum_{q=0}^\infty\sum_{l=\dt_{d_i,0}}^{q+1}
      \frac{(pn+i)^{l+d_i-1}m^{q-l+1}}{(l+d_i-1)!(q-l+1)!}\eleid qi v(m+pn+i+k)\\
 =&(pn+i)\sum_{q=0}^\infty\frac{(m+pn+i)^{q+d_i}}{(q+d_i)!}\eleid qi v(m+pn+i+k)
  =(pn+i)I_{m+pn+i}\cdot v(k).
\end{align*}
Furthermore, for any $k,m,n\in\Z, v\in V$, we have
\begin{align*}
   &(L_mL_n-L_nL_m)\cdot v(k)\\
 =&\(a+n+k+\sum_{i=0}^\infty\frac{m^{i+1}}{(i+1)!}L_i\)
      \(a+k+\sum_{j=0}^\infty\frac{n^{j+1}}{(j+1)!}L_j\)v(m+n+k)\\
  &-\(a+m+k+\sum_{j=0}^\infty\frac{n^{j+1}}{(j+1)!}L_j\)
    \(a+k+\sum_{i=0}^\infty\frac{m^{i+1}}{(i+1)!}L_i\)v(m+n+k)\\
 =&(n-m)(a+k)v(m+n+k)
   +\(n\sum_{j=0}^\infty\frac{n^{j+1}}{(j+1)!}-m\sum\frac{m^{j+1}}{(j+1)!}\)L_jv(m+n+k)\\
  &+\sum_{i=0}^\infty\sum_{j=0}^\infty\frac{m^{i+1}n^{j+1}}{(i+1)!(j+1)!}
      (j-i)L_{i+j}v(m+n+k)\\
 =&(n-m)(a+k)v(m+n+k)\\
   &+\(n\sum_{j=0}^\infty\sum_{i=0}^{j+1}\frac{m^{i+1}n^{j-i}}{(i+1)!(j-i)!}
   -m\sum_{j=0}^\infty\sum_{i=0}^{j+1}\frac{m^{i}n^{j-i+1}}{i!(j-i+1)!}\)L_jv(m+n+k)\\
 =&(n-m)\(a+k+\sum_{j=0}^\infty\frac{(m+n)^{j+1}}{(j+1)!}L_j\)v(m+n+k)
   =(n-m)L_{m+n}\cdot v(k).
\end{align*}
This shows that $V\otimes\C[x^{\pm1}]$ is a $\T$-module.
Since $L_0\cdot v(k)=(a+k)v(k)$ for any $v\in V, k\in\Z$ and
\[V\otimes\C[x^{\pm1}]=\bigoplus_{k\in\Z}V\otimes x^k,\]
we see that $V\otimes\C[x^{\pm1}]$ is a weight $\T$-module.\\
(2) The proof is similar to that of (1) and we omit the details.
\end{proof}

We denote by $\mtvad$ the $\T$-module defined in \eqref{T-action},
and by $\mgvadp$ the $\g$-module defined in \eqref{g-action} respectively.
In the following, we determine their irreducibility.
To do this, we need the following result about the operator $\opwijlm s$ defined in \eqref{defopw}.
\begin{lemma}\label{lem:operation}
Let $l,m\in\Z,1\le i,j\le p-1$ and $V$ be a $\td$-module of annihilating level $(h,q)$ and $v\in V$. Then we have the following equations
\begin{align*}
 &\opwijlm{2h}\cdot v(k)=(-1)^hp^{2h}\binom{2h}h I_h^2v(pl+k+i+j),\\
 &\opwijlm{s}\cdot v(k)=0\quad\text{ for }s>2h
\end{align*}
on the $\T$-module $\mtvad$ for any $k\in\Z$,
and on the $\g$-module $\mgvadp$ provided further that $k,j+k, i+j+k\in\P$.
\end{lemma}
\begin{proof}
 Consider $\opwijlm{s}\cdot v(k)$ as a polynomial in $m$ with coefficients in $V\otimes x^{pl+k+i+j}$.
 Using induction on $s\geq 0$ and \eqref{recformula},
 one can prove that the leading term of $\opwijlm{s}\cdot v(k)$ is
 \[\frac{(-1)^hp^{2h}m^{2h-s}(2h)!}{(h!)^2(2h-s)!}I_h^2v(pl+k+i+j).\]
 Then the lemma follows by taking $s=2h$.
\end{proof}

Now we can prove the main results in this section.
\begin{theorem}\label{thm:irrgmod}
 Let $a\in\C, h,q\in\N,\bmd=(d_0,d_1,\cdots,d_{p-1})\in\{0,1\}^p$ and $P\subseteq\{0,1,\dots,p-1\}$ be such that $h\ge\mindp$.
 Let $V$ be a $\tdp$-module with annihilating level $(h,q)$.
 Then the $\g$-module $\mgvadp$ is irreducible if and only if the $\tdp$-module $V$ is irreducible.
\end{theorem}
\begin{proof}
 The ``only if" part is clear and to prove the ``if" part we assume $W$ is a nonzero $\g$-submodule of $\mgvadp$. Since $W$ is a weight module, we may write
 \[W=\bigoplus_{k\in\P} \tilde{W}_k\otimes x^k,\]
 where $\tilde W_k$'s are subspaces of $V$. Set
 \[\tilde W=\bigcap_{k\in\P} \tilde{W}_k.\]
 We claim that $\tilde W\neq 0$.
 Pick a nonzero vector $v(k)\in W$ for some $k\in\P, v\in\tilde W_k$.
 For any $l,m\in\Z, 1\le i,j\le p-1$ such that $\ov{j+k},\ov{i+j+k}\in P$,
 we have by Lemma \ref{lem:operation} and Remark \ref{rmk:Ih}
 \[\opwijlm{2h}\cdot v(k)=(-1)^hp^{2h}\binom{2h}h I_h^2v(pl+k+i+j)\in
  \tilde W_{pl+k+i+j}\otimes x^{pl+k+i+j}\setminus\{0\},\]
 which implies $I_h^2v\in\tilde W$, proving the claim.

 For any $v\in \tilde W, l,m\in\Z, k\in\P, i\in P\setminus\{0\}$ such that $\ov{k-i}\in P$, we have
 \[ v(k-pm)\in \tilde W_{k-pm}\otimes x^{k-pm}\quad\text{ and }\quad
   v(k-pm-i)\in \tilde W_{k-pm-i}\otimes x^{k-pm-i}.\]
 From the following equations
 \begin{align*}
  &L_{pm}\cdot v(k-pm)=\(a+k-pm+\sum_{j=0}^q\frac{(pm)^{j+1}}{(j+1)!}L_j\)v(k)
  \in \tilde W_k\otimes x^k,\\
  &I_{pm+i}\cdot v(k-pm-i)=\sum_{j=0}^{h-d_i}\frac{(pm+i)^{j+d_i}}{(j+d_i)!}\eleid ji
   v(k)\in \tilde W_k\otimes x^k,
 \end{align*}
 we see that
 \[\sum_{j=0}^q\frac{(pm)^{j+1}}{(j+1)!}L_jv\in\tilde W_k\quad\text{ and }\quad
    \sum_{j=0}^{h-d_i}\frac{(pm+i)^{j+d_i}}{(j+d_i)!}\eleid ji v\in\tilde W_k
   \quad\text{ for any }m\in\Z,1\le i\le p-1. \]
 Thus $L_jv, I_{j+\min\{\bmd,P\setminus\{0\}\}}v\in\tilde W_k$ for all $j\in\N, k\in\P$.
 So $L_jv, I_{j+\min\{\bmd,P\setminus\{0\}\}}v\in\tilde W$ for all $j\in\N$.
 This proves that $\tilde W$ is a nonzero $\T_{\min\{\bmd,P\setminus\{0\}\}}$-submodule of $V$.

 If $\min\{\bmd,P\setminus\{0\}\}>\mindp$, such case happens only when
 $0\in P, d_0=0$ and $d_i=1$ for all $i\in P\setminus\{0\}$,
 or equivalently, $I_0\in\tdp\setminus\T_{\min\{\bmd,P\setminus\{0\}\}}$.
 Notice that $I_0$ acts on $V$ as a fixed scalar since
 $I_0$ is central in $\tdp$ and $V$ is irreducible. So $I_0v\in\tilde W$ and $\tilde W$ is a nonzero $\tdp$-submodule of $V$.
 So $\tilde W=V$ by the irreducibility of $V$, and hence $\mgvadp$ is irreducible.
\end{proof}

\begin{theorem}\label{thm:irrtmod}
 Let $a\in\C, h,q\in\N,\bmd=(d_0,d_1,\cdots,d_{p-1})\in\{0,1\}^p$ and
 let $V$ be a $\td$-module with annihilating level $(h,q)$.
 Then the $\T$-module $\mtvad$ is irreducible if and only if the $\td$-module $V$ is irreducible.
\end{theorem}
\begin{proof}
 The ``only if" is clear. To prove the ``if" part we
 consider $\mtvad$ as the $\g$-module $\mgvadp$ with $P=\{0,1,\dots,p-1\}$,
 which is irreducible by Theorem \ref{thm:irrgmod}.
 So $\mtvad$ is irreducible as a $\T$-module.
\end{proof}

\section{Isomorphism classes}

In this section we determine the isomorphism classes of the modules constructed in the above section for the twisted Heisenberg-Virasoro algebra and gap-$p$ Virasoro algebras.
We first do this for the gap-$p$ Virasoro algebras.

\begin{theorem}\label{thm:isoforgap}
 Let $\bmd,\bme\in\{0,1\}^p, a,b\in\C,h,q,n,t\in\N$ such that $h\geq \mind,n\geq \mine$,
 and $P,Q$ be subsets of $\{0,1,\dots,p-1\}$.
 Let $V$ be an irreducible $\td$-module of annihilating level $(h,q)$ and $W$ be an irreducible $\te$-module of annihilating level $(n,t)$.
 Then $\mgvadp\cong\mgwbeq$ as $\g$-modules if and only if $a-b\in\Z, h=n,q=t,Q=\sigma^{a-b}(P), d_i=e_i$ for all $i\in P-P$ and $V\cong W$ as $\T_{\min\{\bmd,P-P\}}$-modules,
 where $P-P=\left\{\ov{i-j}\mid i,j\in P\right\}$ and $\sigma$ is the permutation $(0,1,\dots,p-1)$ acting on the set $\{0,1,\dots,p-1\}$.
\end{theorem}
\begin{proof}
 The ``if" part follows from the fact that the linear map $\varphi:\mgvadp\longrightarrow\mgwbeq$ defined by
 \[\vf(v(l))=\psi(v)(l+a-b)\quad\text{ for any }v\in V, l\in\Z,\]
 where $\psi:V\longrightarrow W$ is a $\T_{\min\{\bmd,P-P\}}$-module isomorphism,
 is a $\g$-module isomorphism.

 In the next we prove the ``only if" part.
 Let $\varphi:\mgvadp\longrightarrow\mgwbeq$ be a $\g$-module isomorphism.
 For any $v\in V,m\in\P$, write
 \[\vf(v(m))=\sum_{i\in\mathcal Q}w_{i,m}(i),\]
 where $\Q=\{k\in\Z\mid \ov k\in Q\}, w_{i,m}\in W$ and only finitely many $w_{i,m}$'s are nonzero.
 For any $k\in\N$, we have
 \[(a+m)^k\sum_{i\in\mathcal Q}w_{i,m}(i)=\vf\(L_0^k\cdot v(m)\)=L_0^k\cdot \vf\(v(m)\)=
   \sum_{i\in\mathcal Q}(b+i)^k w_{i,m}(i) .\]
 This implies $a-b\in\Z, Q=\sigma^{a-b}(P)$ and
 \[\vf(v(m))=w_{m+a-b,m}(m+a-b)\in W\otimes x^{m+a-b}\quad\text{for any }v\in V,m\in\Z\]
 For any $i\in P$, we may define linear maps $\psi_i:V\longrightarrow V$ by
 \[\vf(v(i))=\psi_i(v)(i+a-b)\quad\text{for any }v\in V.\]
 Clearly, all $\psi_i$'s are linear isomorphisms.

 For any $l,m\in\Z, 1\le i,j\le p-1, v\in V,k\in\P$ such that $j+k\in\P, i+j+k\in \P$, from the equation
 \begin{align*}
  &\opwijlm{2h}\cdot\(\psi_{\ov k}(v)(k+a-b)\)=\opwijlm{2h}\cdot\vf(v(k))
     =\vf\(\opwijlm{2h}\cdot v(k)\)\\
  &=(-1)^h p^{2h}\binom{2h}h \vf\(I_h^2v(pl+i+j+k)\)\in W\otimes x^{pl+i+j+k+a-b}\setminus\{0\}
 \end{align*}
 we see by Lemma \ref{lem:operation} that $h\le n$.
 Apply $\vf^{-1}$ to the above argument and we get $n\le q$. Hence $h=n$.
 Then the above equation becomes
 \[\vf\(I_h^2v(pl+i+j+k)\)=I_h^2\psi_{\ov k}(v)(pl+i+j+k+a-b)\]
 by Lemma \ref{lem:operation}, or (noticing that $I_h$ acts on $V$ as a linear isomorphism)
 \begin{equation*}
  \vf\(v(pl+i+j+k)\)=I_h^2\psi_{\ov k}\(I_h^{-2}v\)(pl+i+j+k+a-b)
 \end{equation*}
 for any $l\in\Z,v\in V, 1\le i,j\le p-1,k\in\P$ such that $j+k\in\P, i+j+k\in \P$.
 Take $l\in\Z$ such that $pl+i+j+k=\ov{i+j+k}$ and we get
 \[\psi_{\ov{i+j+k}}(v)=I_h^2\psi_{\ov k}\(I_h^{-2}v\),\]
 which implies that $\psi_i=\psi_j$ for all $i,j\in P$.
 Set $\psi=\psi_i$ for any $i\in P$.
 Then we have
 $$\vf(v(l))=\psi(x)(l+a-b)\quad\text{ for any }v\in V, l\in\P.$$
 From $\vf(L_{pm}\cdot v(j-pm))=L_{pm}\cdot\vf( v(j-pm))$ with $m\in\Z, v\in V, j\in P$,
 we have
 \[\psi\(\(a+j-pm+\sum_{k=0}^q\frac{(pm)^{k+1}}{(k+1)!}L_k\)v\)=
   \(a+j-pm+\sum_{k=0}^t\frac{(pm)^{k+1}}{(k+1)!}L_k\)\psi(v).\]
 This implies that $q=t$ and
 \[\sum_{k=0}^q\frac{(pm)^{k+1}}{(k+1)!}(\psi(L_kv)-L_k\psi(v))=0
 \quad\text{ for any }m\in\Z.\]
 So $\psi(L_kv)=L_k\psi(v)$ for any $v\in V,k\in\N$.

 From $\vf(I_{pm+i}\cdot v(j-pm))=I_{pm+i}\cdot \vf(v(j-pm))$ with $m\in\Z, 1\le i\le  p-1, j\in\ P$ and $\ov{i+j}\in P$, we see that
 \[\sum_{k=0}^{h-d_i}\frac{(pm+i)^{k+d_i}}{(k+d_i)!}\psi(\eleid ki v)=
   \sum_{k=0}^{h-e_i}\frac{(pm+i)^{k+e_i}}{(k+e_i)!}I_{k+e_i} \psi(v)\quad
   \text{ for any }m\in\Z.\]
 Consider the both sides of the above equation as polynomials in $m$ with coefficients in $W$, and we get that $d_i=e_i$ for all $i\in P-P$, and $\psi(\eleid kiv)=\eleid ki\psi(v)$ for any $k\geq 0$.
 Equivalently,
 \[\psi(I_kv)=I_k\psi(v)\quad\text{ for any }v\in V, k\geq \min\{\bmd,P-P\}.\]
 This proves that $\psi$ is a $\T_{\min\{\bmd,P-P\}}$-module isomorphism and we complete the proof.
\end{proof}

Now, we determine the isomorphism classes of the modules $\mtvad$ for $\T$.
\begin{theorem}
 Let $\bmd,\bme, a,b,h,q,n,t,V,W$ be as in Theorem \ref{thm:isoforgap}.
 Then $\mtvad\cong\mtwbe$ as $\T$-modules if and only if $a-b\in\Z, h=n,q=t, \bmd= \bme$ and $V\cong W$ as $\td$-modules.
\end{theorem}
\begin{proof}
 Notice that the statement that $\mtvad\cong\mtwbe$ as $\T$-modules implies that the $\g$-modules $\mgvadp$ and $\mgwbeq$ are isomorphic  with $P=Q=\{0,1,\dots,p-1\}$.
 Then the proof follows from Theorem \ref{thm:isoforgap}.
\end{proof}

\begin{remark}
 Let $a,b,c\in\C$ and $\C v$ be a $\T_0$-module with $L_0v=bv, I_0v=cv,L_kv=I_kv=0$ for $k>0$. \\
 (1) Then there is a $\T$-module structure on $v\otimes\C[x^{\pm1}]$ with actions
 \[L_m\cdot v(k)=(a+k+bm)v(m+k),\ I_m\cdot v(k)=cv(m+k)\quad\text{ for any }m,k\in\Z.\]
 The resulting $\T$-module is the module $A_{a,b,c}$ of intermediate series (see \cite{LJ}).\\
 (2) Let $P$ be a subset of $\{0,1,\dots,p-1\}$ and $(d_1,d_2,\dots,d_{p-1})\in\{0,1\}^{p-1}$ such that
 $$\min\{d_1,d_2,\dots,d_{p-1}\}=0 \quad\text{ and }\quad
      \ov{i+j}\in P \text{ provided that }j\in P, d_i=0.$$
 Let $F=(F_{i,j})$ be the $(p-1)\times p$ matrix with indexes $1\le i\le p-1, 0\le j\le p-1$ and $F_{i,j}=c\dt_{d_i,0}\dt_{j,P}$.
 Then there is a $\g$-module structure on $v\otimes \C[P]$ with actions
 \[L_{pm}\cdot v(k)=(a+k+bpm)v(m+k),\ I_{pm+i}\cdot v(k)=F_{i,\ov k}v(pm+i+k)\quad\text{ for }m\in\Z,k\in\P,1\le i\le p-1.\]
 The resulting $\g$-module is the module $V(a,b,F)$ of intermediate series constructed in \cite{Xu}.
 \end{remark}

\section{The modules are new}\label{sec5}

In this section we demonstrate that the modules constructed in Section \ref{sec3} are new modules for both $\T$ and $\g$.

The $\T$-modules $\mtvad$ generalize the ones $M(V,a)$ in \cite{CHS}.
As far as we know, these modules are new simple weight modules for $\T$.

For $p>2$, the only kind of weight modules over the gap-$p$ Virasoro algebra $\g$ studied before this paper to the best of our knowledge are the Harish-Chandra modules.
Notice that any irreducible $\T_0$ (or $\T_1$)-module $V$ must be infinite dimensional if $\dim V>1$.
Therefore, the $\g$-modules $\mgvadp$ with $\dim V>1$ are new simple weight modules for $\g$ with infinite dimensional weight spaces.

In the rest of this section, we set $p=2$, in which case $\g$ is isomorphic to the mirror Heisenberg-Virasoro algebra.

\begin{definition}\label{defmirror}
 The mirror Heisenberg-Virasoro algebra is a complex Lie algebra $\mathbb M$ with a basis
 \[\left\{d_i,h_{r},\bm c, \bm l\mid i\in\Z,r\in\Z+\frac12\right\}\]
 and commutation relations
 \begin{align*}
 &[d_m,d_n]=(m-n)d_{m+n}+\frac{m^3-m}{12}\bm c\dt_{m+n,0};\\
 &[d_m,h_r]=-rh_{m+r};\quad [h_r,h_s]=r\bm l\dt_{r+s,0},\quad [\bm c,\mathbb M]=[\bm l,\mathbb M]=0,
 \end{align*}
 where $m,n\in\Z,r,s\in\Z+\frac12$.
\end{definition}
One can see easily that the linear map defined by
\[L_{2m}\mapsto 2 d_{-m},\quad I_{2m+1}\mapsto 2 h_{-m-\frac 12},\quad
   C_{0}\mapsto 4\bm c,\quad  C_{1}\mapsto -2\bm l\]
for $m\in\Z$, is a Lie algebra isomorphism from $\g$ to $\M$.

Since the modules $\mgvadp$ for $\g$ with $\dim V>1$ have infinite dimensional weight spaces,
they are not the Harish-Chandra modules classified in \cite{LPXZ},
and they are certainly not the non-highest-weight-type restricted modules studied in \cite{CG1,Gao} and \cite{TYZ}.

In \cite{GZ}, tensor products of irreducible highest weight modules and modules of intermediate series over $\M$ were studied.
These modules have infinite dimensional weight spaces.
We recall highest weight modules and modules of intermediate series over $\M$ in the following.

Let $\al,\be\in\C,\gm\in\C^\times$ and let $Q$ be a nonempty subset of $\{0,1\}$.
Then the vector space $A(\al,\be,\gm,Q)=\spanc{ v_{r+2k}\mid r\in Q, k\in\Z}$ carries a $\g$-module structure with actions
\begin{align*}
 &L_{2m}v_{r+2k}=(\al+2\be m+2k+r)v_{r+2m+2k};\quad   \bm c v_k=\bm lv_k=0;\\
 &I_{2m+1}v_{r+2k}=\begin{cases}
               0 & \text{ if }Q=\{0\}\text{ or }Q=\{1\};\\
      v_{2m+2n+1}& \text{ if }Q=\{0,1\}, r=0;\\
      \gm v_{2m+2n+2}& \text{ if }Q=\{0,1\}, r=1
      \end{cases}
\end{align*}
for $m,n,k\in\Z, r\in Q$.
The $\g$-module $A(\al,\be,\gm,Q)$ is reducible if and only if
$Q\neq\{0,1\}, \al\in2\Z$ and $\be\in\{ 0,1\}$.
We denote by $A'(\al,\be,\gm,Q)$ the unique irreducible subquotient of $A(\al,\be,\gm,Q)$.
Recall the operator $\opwijlm s$ in \eqref{defopw} and it is easy to check that
\begin{equation}\label{w-action}
 \Omega_{l,m}^{(1,1,s)}u=0\quad\text{for any }l,m\in\Z,s\in\N, u\in A'(\al,\be,\gm,Q).
\end{equation}
The following result is from \cite[Theorem 3.9]{LPXZ} and \cite[Theorem 4.3]{Xu}.
\begin{lemma}
 Any irreducible $\g$-module of intermediate series is isomorphic to $A'(\al,\be,\gm,Q)$ for some $\al,\be\in\C,\gm\in\C^\times,\emptyset\neq Q\subseteq\{0,1\}$.
\end{lemma}

Set $\g_+=\spanc{L_{2m},I_{2m+1},C_0,C_1\mid m\in\N}$.
For $c,h,l\in\C$, let $\C\bm1$ be the 1-dimensional $\g_+$-module with
\[L_0\bm1=h\bm1,\quad C_0\bm1=c\bm 1,\quad C_1\bm1=l\bm 1,\quad
   L_{2m}\bm1=I_{2m-1}\bm1=0\quad\text{ for any }m\in\Z_+.\]
The Verma module over $\g$ is defined to be
\[M(c,h,l)=\U(\g)\otimes_{\U(\g_+)}\C\bm 1.\]
Denote by $J$ the unique maximal submodule of $ M(c,h,l)$ and set
$$L(c,h,l)=M(c,h,l)/J,$$
which is called the highest weight $\g$-module of level $(c,h,l)$.
Denote by $\im$ the image of $\bm1$ in $L(c,h,l)$.

Irreducibility of the tensor product $L(c,h,l)\otimes A'(\al,\be,\gm,Q)$ was given in \cite[Theorem 3.11]{GZ}.
In the following proposition we prove that the modules studied in Section 3 are not isomorphic to any of these tensor product modules.
Thus we conclude that a new class of irreducible weight modules for the mirror Heisenberg-Virasoro algebra are constructed.
\begin{proposition}\label{noniso}
 Let $a\in\C, \bmd=(d_0,d_1)\in\{0,1\}^2, P\subseteq\{0,1\}$ and let $V$ be an irreducible $\tdp$-module of annihilating level $(n,t)$.
 Then the $\g$-module $\mgvadp$ is not isomorphic to $L(c,h,l)\otimes A'(\al,\be,\gm,Q)$ for any $c,h,l,\al,\be,\gm\in\C,\emptyset\neq Q\subseteq\{0,1\}$.
\end{proposition}
\begin{proof}
 Take $l,m\in\Z_+$ such that $l-m-n>0$.
 Then it is easy to see that $I_{2m+2k+1}\im=I_{2l-2m-2k+1}\im=0$ for any $0\le k\le n$.
 For any $u\in A'(\al,\be,\gm,Q)$, it follows from \eqref{w-action} that
 \begin{align*}
  &\Omega_{l,m}^{(1,1,2n)}(\im\otimes u)=
     \sum_{k=0}^{2n}(-1)^{2n-k}\binom{2n}k I_{2l-2m-2k+1}I_{2m+2k+1}(\im\otimes u)\\
  &\ =\im\otimes\(\sum_{k=0}^{2n}(-1)^{2n-k}\binom{2n}k I_{2l-2m-2k+1}I_{2m+2k+1} u\)
    =\im\otimes\(\Omega_{l,m}^{(1,1,2n)} u\)=0.
 \end{align*}
 On the other hand, the operator $\Omega_{l,m}^{(1,1,2n)}$ acts on $\mgvadp$ as a linear isomorphism by Lemma \ref{lem:operation}.
 So $\mgvadp$ can not be isomorphic to any $L(c,h,l)\otimes A'(\al,\be,\gm,Q)$.
\end{proof}

\section{Some irreducible non-weight modules}

The purpose of this section is to construct a new class of non-weight modules with trivial central actions for $\T$ and $\g$ by using the twisting technique in \cite{LGZ} to the modules $\mtvad$ and $\mgvadp$.

Let $f=\sum_ia_ix^i\in\C[x^{\pm1}]$. There is a new $\T$-module structure on $\mtvad$
defined by
\[L_m\circ v=\(L_m+\sum_i a_iI_{m+i}\)\cdot v;\quad I_m\circ v=I_m\cdot v\quad
  \text{for any }m\in\Z,v\in\mtvad.\]
The resulting $\T$-module is denoted by $\mftvad$.

\begin{theorem}\label{thm:twisttmod}
 Let $\bmd\in\{0,1\}^p, a\in\C$ and $V$ be an irreducible $\td$-module of some fixed annihilating level.
 Then for any $f\in\C[x^{\pm1}]$, the $\T$-module $\mftvad$ is irreducible.
\end{theorem}
\begin{proof}
 Denote
 $$\al=\sum_i\frac{a_i}i I_i\in\T\quad\text{and}\quad\theta_\al=\exp(\text{ad} \al).$$
 Then one can see that the $\T$-module $\mftvad$ is the $\theta_\al$-twist of $\mtvad$, i.e. we have
 \[ u\circ v=\theta_\al(u)\cdot v\quad\text{for any}\quad u\in\T,v\in\mtvad.\]
 Then the irreducibility of $\mftvad$ is equivalent to that of $\mtvad$,
 which follows from Theorem \ref{thm:irrtmod}.
\end{proof}

Let $g=\sum_{i\in\Z\setminus p\Z}a_ix^i\in\C[x^{\pm1}]$. There is a new $\g$-module structure on $\mgvadp$ defined by
\[L_{pm}\circ v=\(L_{pm}+\sum_{i\in\Z\setminus p\Z} a_iI_{pm+i}\)\cdot v;
  \quad I_{pm+j}\circ v=I_{2m+j}\cdot v\]
for any $1\le j\le p-1,m\in\Z,v\in\mgvadp$.
We denote the resulting $\g$-module by $\mggvadp$.

\begin{theorem}\label{thm:twistgmod}
 Let $\bmd\in\{0,1\}^p, a\in\C, P\subseteq\{0,1,\dots,p-1\}$ and $V$ be an irreducible $\tdp$-module of some fixed annihilating level.
 Then for any $g=\sum_{i\in\Z\setminus p\Z}a_ix^i\in\C[x^{\pm1}]$, the $\g$-module $\mggvadp$ is irreducible.
\end{theorem}
\begin{proof}
 The proof is similar to that of Theorem \ref{thm:twisttmod} and we omit the details.
\end{proof}

\begin{remark}
 (1) The $\T$-modules $\mftvad$ generalize the ones studied in \cite[Section 5]{CHS}, hence are not isomorphic to the restricted modules studied in \cite{CG1,Gao} and \cite{TYZ} (see \cite[Proposition 5.1]{CHS}).
 Thus they are new non-weight modules for $\T$.\\
 (2) By a similar proof to Proposition \ref{noniso}, one can show that
 the $\g$-modules $\mggvadp$ are not isomorphic to the restricted modules studied in \cite{GX} if $p>2$,
 and not isomorphic to the restricted modules studied in \cite{TYZ} if $p=2$.
 Thus they are new non-weight modules for $\g$.
\end{remark}

\bigskip

\textbf{Acknowledgments:}
C. Xu is supported by the National Natural Science Foundation of China (No. 12261077).

\end{document}